\documentclass[12pt]{amsart}
\usepackage{amscd, bbm, mathtools}
\usepackage{amssymb,mathabx}
\usepackage{todonotes}
\usepackage{mathrsfs,wasysym}
\usepackage{amsmath}
 \usepackage[margin=24mm]{geometry}
\usepackage{hyperref}
\hypersetup{colorlinks=true, linkcolor = blue, citecolor=purple}

\makeatletter
\@addtoreset{equation}{section}

\numberwithin{equation}{section}

\newtheorem{Theorem}{Theorem}[section]
\newtheorem*{Theorem*}{Theorem}
\newtheorem*{Corollary*}{Corollary}

\newtheorem*{Lemma*}{Lemma}
\newtheorem{Lemma}[Theorem]{Lemma}
\newtheorem{Proposition}[Theorem]{Proposition}
\newtheorem{Corollary}[Theorem]{Corollary}

\theoremstyle{definition}
\newtheorem{Definition}[Theorem]{Definition}

\theoremstyle{remark}
\newtheorem{Remark}[Theorem]{Remark}
\newtheorem{Remarks}[Theorem]{Remarks}
\newtheorem*{Remark*}{Remark}

\newbox\squ  
\setbox\squ=\hbox{\vrule width.6pt
   \vbox{\hrule height.6pt width.4em\kern1ex
         \hrule height.6pt}%
   \vrule width.6pt}

\newcommand{\N}{\mathbb{N}}
\newcommand{\Z}{\mathbb{Z}}

\newcommand{\A}{\mathbb{A}}
\renewcommand{\P}{{\mathbb{P}}}

\newcommand{\g}{\mathfrak{g}}

\newcommand{\Gm}{\mathbb{G}_m}

\newcommand{\m}{\mathsf{m}}

\newcommand{\Gr}{\operatorname{Gr}}

\renewcommand{\O}{\mathcal{O}}

\newcommand{\B}{\mathscr{B}}

\newcommand{\ad}{\operatorname{ad}}

\newcommand{\Lie}{\operatorname{Lie}}
\newcommand{\Ker}{\operatorname{Ker}}

\newcommand{\id}{\operatorname{id}}

\newcommand{\In}{\operatorname{in}}
\newcommand{\Hom}{\operatorname{Hom}}
\newcommand{\GL}{\operatorname{GL}}
\newcommand{\SL}{\operatorname{SL}}

\newcommand{\Spec}{\operatorname{Spec}}

\newcommand{\Ind}{\operatorname{Ind}}

\newcommand{\Char}{\operatorname{char}}
\newcommand{\alg}{\operatorname{-alg}}

\newcommand{\Lr}{\mathcal{L}_{\operatorname{ring}}}

\title[Smoothness of stabilisers in generic characteristic]{Smoothness of stabilisers in generic characteristic}

\author[Ben Martin]{Ben Martin}
\address
{Institute of Mathematics,
University of Aberdeen,
Aberdeen AB24 3FX,
United Kingdom}
\email{B.Martin@abdn.ac.uk}

\author{David I. Stewart}
\address{Department of Mathematics,
University of Manchester,
Manchester,
United Kingdom}
\email{david.i.stewart@manchester.ac.uk}

\author{Lewis Topley}
\address{Department of Mathematical Sciences,
University of Bath,
Bath BA2 7AY,
United Kingdom}
\email{lt803@bath.ac.uk}


\begin{document}

\begin{abstract}
Let $R$ be a commutative unital ring. Given a finitely presented affine $R$-group scheme $G$ acting on a separated scheme $X$ of finite type over $R$, we show that there is a prime $p_0$ such that for any $R$-algebra $k$ which is an algebraically closed field of characteristic $p\geq p_0$, the centraliser in $G_k$ of any closed subscheme of $X_k$ is smooth. When $X$ is not necessarily separated we show similarly that for any closed subscheme $Y \subseteq X$ there is a $p_1$ depending on $Y$ such that when $k$ has characteristic $p \geq p_1$ the normaliser of $Y$ in $G_k$ is smooth. We prove these results using the Lefschetz principle together with careful application of Gr\"obner basis techniques, and using a suitable notion of the complexity of an action.

We apply our results to demonstrate that the Kostant-Kirillov-Souriau theorem holds for Lie algebras of algebraic groups in large positive characteristics. In particular, every such Lie algebra decomposes as a disjoint union of symplectic varieties, each of which is a coadjoint orbit.
\end{abstract}

\maketitle

\section{Introduction}
Let $R$ be a commutative unital ring. By an {\em $R$-field} we mean an $R$-algebra that is also a field and by an {\em algebraic $R$-group} we mean a finitely presented affine $R$-group scheme. We prove the following.

\begin{Theorem}
\label{thm:smoothcent}
 Let $G$ be an algebraic $R$-group and let $X$ be a separated $G$-scheme finitely presented and of finite type over $R$. Then there exists $p_0 \in \N$ such that whenever $k$ is an  $R$-field of characteristic $p\geq p_0$, the centraliser $C_{G_k}(Y)$ is smooth for every closed subscheme $Y$ of $X_k$.
\end{Theorem}

\begin{Theorem}
\label{thm:smoothnorm}
 Let $G$ be a an algebraic $R$-group and let $X$ be a $G$-scheme finitely presented and of finite type over $R$. Let $Y$ be a closed subscheme of $X$.  Then there exists $p_1 \in \N$ such that whenever $k$ is an $R$-field of characteristic $p\geq p_1$, the normaliser $N_{G_k}(Y_k)$ is smooth.
\end{Theorem}

Theorem~\ref{thm:smoothcent} uses the hypothesis that $X$ is separated in order to infer that the centralisers are closed subschemes of $G$ (see \cite[I.2.6]{Jan03}). Note also that the lower bound of $p_0$ in Theorem~\ref{thm:smoothcent} for centralisers does not depend on $Y$; the bound in Theorem~\ref{thm:smoothnorm} for normalisers does, however, depend on $Y$, and this dependence cannot be removed: see Remark~\ref{rem:normdpndce}.

Natural examples of algebraic groups over rings abound, of course. The split reductive groups are all $\Z$-defined \cite{Dem63}; and so too are the subgroups of reductive groups normalised by a split maximal torus---so-called \emph{subsystem subgroups}. This class includes all parabolic subgroups of reductive groups.

There are several known special cases of the theorem. One of the most influential in Lie theory occurs when $G$ is split reductive and $X$ is either $G$ itself or its Lie algebra ${\mathfrak g}$, on which $G$ acts by the relevant adjoint action. Then it is well-known that the group $G$, the algebra ${\mathfrak g}$, and the adjoint action are defined over $\Z$ (see \cite[II.1.1]{Jan03}) and the centralisers of single elements of $X(\bar k)$ are smooth whenever $p$ is a very good prime for $G$. This follows (see \cite[Thm.~1.3(a)]{BMRT07}) from work of Richardson \cite{Ric67}, who used the notion of a reductive pair to give an elegant proof for very good $p$ that the number of unipotent and nilpotent orbits of $G$ is finite. This smoothness result was generalised in \cite[Thm.~1.2]{BMRT07} to cover arbitrary subgroups of $G(\bar k)$ and subalgebras of $\Lie(G)$. The hypotheses were further weakened in \cite{Her13}. Normalisers, while much less well-behaved, were thoroughly considered in \cite{HS16}, where it was shown that (necessarily large) bounds on the characteristic exist, depending on the root system, which ensure the normalisers of subspaces of the Lie algebras of reductive groups are smooth. Through the classification of nilpotent and unipotent orbits, these results have found applications in developing the subgroup structure of simple algebraic groups and subalgebra structure of their Lie algebras; recent examples include \cite{LT18} and \cite{PSMax}.

One new feature of our work is that we move beyond the affine case: our results apply not just to affine varieties but to quasi-projective varieties (and other varieties of finite type).  Here is an application.  The second author computed explicitly in \cite{SteComp} the orbits of exceptional groups on their Lie algebras, determining when centralisers and stabilisers of lines are smooth for minimal induced (or dual-Weyl) modules; non-smoothness occurs only in characteristics $2$ or $3$.  This motivated [\textit{op.~cit.}, Question 1.4], to which our theorem provides the following strong answer.

\begin{Corollary}
\label{cor:module}
 Let $G$ be an algebraic $R$-group and let $V$ be a $G$-module which is finitely generated and projective as an $R$-module. Then there is a prime $p_V$ such that whenever $k$ is an  $R$-field of characteristic $p\geq p_V$, the centraliser $C_{G_k}(W)$ and normaliser $N_{G_k}(W)$ are smooth for any $k$-subspace $W$ of $V_k$.
\end{Corollary}

\noindent The smoothness of centralisers follows immediately from Theorem~\ref{thm:smoothcent}.  To prove smoothness of normalisers one can apply Theorem~\ref{thm:smoothnorm}.  Alternatively, one can apply Theorem~\ref{thm:smoothcent} to $X=  \Gr_m(V)$, where $\Gr_m(V)$ denotes the Grassmannian of $m$-generated submodule of $V$: the idea is that the stabiliser of a subspace is the centraliser of the corresponding point in the Grassmannian---for details, see Section~\ref{subsec:proofs}.  The same conclusion does not hold if we weaken the hypothesis and consider all $G_k$-modules $V$ of bounded dimension (see Remark \ref{nonsmoothex}), but it does if $G_k$ is reductive and we bound the weights of the $G_k$-module $V$ in an appropriate sense: see Proposition~\ref{prop:repns}. 

In the final section of this paper we apply our main theorem to prove a modular analogue of the Kostant-Kirillov-Souriau (KKS) theorem from symplectic geometry. Their theorem states that if $G$ is a complex algebraic group with Lie algebra $\g$ then the symplectic leaves of the Poisson variety $\g^*$ are precisely the coadjoint orbits. When $G$ is semisimple, this result leads to a classification of symplectic homogeneous $G$-varieties, since they all arise as finite covers of coadjoint orbits. We refer the reader to \cite[\textsection~IV.7]{GS77} for a detailed background of the theory. When $k$ is an algebraically closed field of characteristic $p > 0$ one can ask whether the Poisson variety $\g^*$ decomposes into a disjoint union of symplectic $G$-homogeneous subvarieties. In general the answer is negative, however we show that the KKS theorem holds whenever the characteristic is sufficiently large (Theorem~\ref{T:leaves}).

Let us say some words on the proofs of our main results.  The central idea is to combine the Lefschetz principle from first-order model theory with Gr\"obner basis techniques.  This approach was suggested in \cite{Sch00} as a method to solve certain problems in algebraic geometry, but to our knowledge the current paper is the first place it has been carried out in practice.  Roughly speaking, the Lefschetz principle says that a first-order property that holds for algebraically closed fields of characteristic 0 holds for algebraically closed fields of large enough characteristic.  In order to apply it, one needs suitable notions of complexity for schemes, morphisms and group actions.  This allows us to work with honest first-order sentences without needing parameters from the field: rather than considering a fixed $G$ and a fixed $X$, we quantify over all $G$ and $X$ of bounded complexity (a similar trick was used in \cite{MSTT18}).  A crucial tool is Lemma~\ref{idealmem}, which allows us to give a bound on complexities arising from certain ideal membership problems.  We find, surprisingly, that although certain complexities need to be constrained, others do not: for example, we prove variations of Theorems~\ref{thm:smoothnorm} and \ref{thm:smoothcent} which hold for schemes $X$ that are not of finite type (see Corollaries~\ref{cor:smoothnorm} and \ref{cor:smoothcent}). 

The paper is set out as follows.  In Section~\ref{S:preliminaries} we recall the language of Hopf algebras and the Lefschetz principle from model theory, which we use to pass information from characteristic zero to large positive characteristics. Section~\ref{S:smoothcentralisers} is the heart of the paper. Here we recall the theory of Gr{\"o}bner bases which allows us to express a criterion for smoothness in terms of first-order sentences in the language of rings.  We recall the definition of complexity for schemes and morphisms (Section~\ref{subsec:complexity}).  We introduce the notion of a $d$-bounded Hopf quadruple, and in Theorem~\ref{gensmooth} we prove that such Hopf algebras correspond to smooth algebraic groups in large characteristics $p > p_0(d)$. In Section~\ref{S:normalandcent} we provide the proofs of the main theorems, using the functor-theoretic descriptions of centralisers and normalisers and the notion of $G$-complexity (Definition~\ref{defn:Nbdd}).
Finally in Section~\ref{S:KKS} we provide a short proof of our modular version of the KKS theorem from symplectic geometry.

\bigskip

\noindent {\bf Acknowledgements:} The research of the second author is funded by Leverhulme Trust Research Project Grant number RPG-2021-080. The third author is grateful for the support of EPSRC grant EP/N034449/1, as well as funding from the UKRI, grant number MR/S032657/1. We are grateful to Sylvy Anscombe and Charlotte Kestner for helpful discussions.

For the purpose of open access, the author has applied a Creative Commons Attribution (CC BY) [or other appropriate open licence] licence to any Author Accepted Manuscript version arising from this submission.

\section{Preliminaries}
\label{S:preliminaries}
Throughout the paper we consider a fixed commutative unital ring $R$.

\subsection{Schemes, group schemes and Hopf algebras}
\label{ss:groupschemes}
We take the functorial approach to schemes, as per \cite{DG:1970} and \cite{Jan03}. Thus for an $R$-algebra $Q$ we think of $\Spec_R(Q)$ as the functor $\Hom_{R\text{-Alg}}(Q,-)\colon \underline{R\text{-Alg}}\to\underline{\text{Set}}$.  An \emph{$R$-functor} is a functor $X\colon \underline{R\text{-Alg}}\to\underline{\text{Set}}$; we call an $R$-functor $X$ \emph{affine} if it is isomorphic to $\Spec_R(R[X])$ for some $R$-algebra $R[X]$. If $Q$ is an $R$-algebra then $X_Q$ denotes the $Q$-functor obtained from $X$ by base extension.  We say a subfunctor $Y$ of an $R$-functor $X$ is \emph{open} if for every $R$-algebra $Q$ and natural transformation $\beta\colon \Spec_R Q\to X$, the subfunctor $\beta^{-1}(Y)$ of $\Spec_R Q$ is an open subfunctor of $\Spec_R Q$, and is {\em closed} if for every $R$-algebra $Q$ and natural transformation $\beta\colon \Spec_R Q\to X$, the subfunctor $\beta^{-1}(Y)$ of $\Spec_R Q$ is a closed subfunctor of $\Spec_R Q$. Then $X$ is a \emph{scheme} (or an \emph{$R$-scheme}) if it is \emph{local} in the sense of \cite[I.1.8]{Jan03} and admits a decomposition $X=\bigcup_{i\in \mathbb I}X_i$ for some indexing set $\mathbb I$, where the $X_i$ are open affine subfunctors of $X$. 

If $Q$ is an $R$-algebra and $X$ is an $R$-scheme then $X_Q$ is an $Q$-scheme \cite[I.1.10]{Jan03}.  We say $X$ is of \emph{finite type} if $\mathbb I$ is finite and each $R[X_i]$ is finitely generated over $k$.  We say $X$ is \emph{locally finitely presented} if each $R[X_i]$ is a finitely presented $R$-algebra.  If $R$ is Noetherian (e.g., a field) then any finitely generated $R$-algebra is finitely presented, so in this case any $R$-scheme of finite type is locally finitely presented.  We do not assume that all schemes are separated; recall that a scheme is \emph{separated} if that the diagonal map $D:X\to X\times X$ is an embedding. By an embedding of schemes we mean a closed immersion.

An {\em affine $R$-group scheme} $G$ is a functor from $\underline{R\text{-Alg}}$ to \underline{Grp}, which as a functor to $\underline{\text{Set}}$, is naturally equivalent to one of the form $\Spec_R(R[G])$ for some finitely generated $R$-algebra $R[G]$.  We consider only the case where $R[G]$ is finitely presented; in keeping with \cite[I.2.1]{Jan03}, we then call $G$ an {\em algebraic $R$-group}.  We do not assume that algebraic $R$-groups are smooth.  There is a natural notion of a closed subgroup of $G$ ({\em loc.\ cit.}).  The archetypal example of an algebraic $R$-group is $\GL_d$, which is also an example of a split reductive group. 

A Hopf $R$-algebra consists of data $(R[G], \Delta, \sigma, \epsilon)$ where $R[G]$ is an $R$-algebra, and there are $R$-algebra homomorphisms $\Delta \colon  R[G] \to R[G] \otimes_R R[G]$, $\sigma \colon  R[G] \rightarrow R[G]$ and $\epsilon \colon  R[G] \to R$  satisfying the dual of the group axioms \cite[I.2.3(1--3)]{Jan03}:

\begin{eqnarray}(\Delta\otimes\id)\circ\Delta=(\id\otimes\,\Delta)\circ \Delta\label{assoc}\\
(\epsilon\,\bar\otimes\id)\circ\Delta=(\id\bar{\otimes}\,\epsilon)\circ\Delta\label{ident}\\
(\sigma\,\bar\otimes\id)\circ\Delta=\bar\epsilon=(\id\bar\otimes\,\sigma)\circ\Delta\label{inv}.\end{eqnarray}
\noindent Here the symbol $\varphi \bar \otimes \psi$ denotes the tensor product of the maps $\varphi, \psi$, followed by the natural multiplication map $R[G] \otimes R[G] \to R[G]$, and $\bar\epsilon$ denotes $\epsilon$ followed by the inclusion of $R$ in $R[G]$. Hence by definition, the category of algebraic $R$-groups is the opposite category to the category of finitely presented Hopf algebras over $R$. 

The Lie algebra $\Lie(G)$ of an algebraic $R$-group $G$ is defined to be the $R$-module of all $R$-linear maps $I / I^2 \to R$ where $I = \Ker(\epsilon)$; in other words it is $\ker \left(G(R[\varepsilon]/(\varepsilon^2))\to G(R)\right)$ where $R[\varepsilon]/(\varepsilon^2)$ is the algebra of dual numbers and the map takes $\epsilon$ to $0$. Following \cite[I.7.7(3)]{Jan03} one obtains a natural $R$-linear Lie bracket on $\Lie(G)$ induced by the comultiplication $\Delta$. Every morphism of algebraic $R$-groups induces a natural $R$-linear homomorphism of their Lie algebras.

When $G$ is an algebraic $R$-group and $k$ is any $R$-algebra, we can consider the base change $G_k$, which is an algebraic $k$-group. To see this concretely, suppose $$R[G] \cong R[x_1,...,x_n]/ (g_1,...,g_m).$$  We have an obvious map $\omega \colon  R[x_1,...,x_n] \rightarrow k[x_1,...,x_n]$ and we see that
\begin{eqnarray}
\label{e:presentingbasechage}
G_k \cong \Spec_k \left(k[x_1,...,x_n] / (\omega(g_1),..., \omega(g_m))\right).
\end{eqnarray}

An action of $G$ on an $R$-scheme $X$ is a natural transformation $\alpha\colon G\times X\to X$, such that $\alpha(Q)\colon G(Q)\times X(Q)\to X(Q)$ is a group action for all $R$-algebras $Q$. If $X$ is an $R$-module (see \cite[I.2.2]{Jan03}) and $G(Q)$ acts through $Q$-linear transformations of $X(Q)$, then we say $X$ is a $G$-module. Note that if $X$ is affine, then we get a coaction map of $R$-algebras $\Delta_X\colon R[X]\to R[X]\otimes R[G]$. In case $X$ is a $G$-module, this is the \emph{comodule map} of \cite[I.2.8]{Jan03}.

Let $V$ be a $G$-module such that $V$ is finitely generated and projective as an $R$-module.  We may regard $V$ as an affine $G$-scheme of finite type over $R$ with co-ordinate ring the symmetric algebra $S[V^*]$ \cite[I.2.2 and I.2.7]{Jan03}.  This construction commutes with base extension.

\subsection{Model theory and the Lefschetz principle}
\label{ss:modeltheory}

In this paper we use the Lefschetz principle to deduce statements about algebras over algebraically closed fields of large positive characteristic from the corresponding statements in characteristic zero.  In doing so we pursue a theme from our earlier work \cite{MSTT18}, in which we proved a version of the first Kac-Weisfeiler Theorem for representations of modular Lie algebras using the Lefschetz principle.  A detailed introduction to the principle can be read in \cite{Ma02} for example, and a concise overview may be found in \cite[\textsection 2.1]{MSTT18}.

For future reference we state the Lefschetz principle (the version we use is taken from \cite[Cor.\ 2.4]{MSTT18}). The \emph{language of rings $\Lr$} is the collection of first-order formulas that
can be built from the symbols $\{\forall, \exists, \vee, \wedge, \neg, +,  $ $ -, \times, 0, 1, =\}$ along with
arbitrary choice of variables.  A {\em sentence} is a formula containing no free variables.  For example, the formula $(\exists y) y^2= x$ is not a sentence because the variable $x$ is free, but by quantifying over $x$ we can obtain a sentence: e.g., $(\forall x)(\exists y) y^2= x$.  Given a ring $R$, a sentence is either true or false for that ring: for instance, the sentence $(\exists y_1)(\exists y_2) (( y_1 \ne y_2) \wedge (y_1^2 = y_2^2))$ is true in every field of characteristic $p> 2$, but is false in every ring of characteristic 2.

A {\em theory} is a set of sentences.  A ring $\mathfrak{R}$ is a {\em model} of a theory $T$ if every sentence belonging to $T$ is true in $\mathfrak{R}$: for instance, if $p$ is either 0 or prime then {\sf AC}$_p$ is the set of sentences that are true in every algebraically closed field of characteristic $p$, and any algebraically closed field of characteristic $p$ is a model of $T$.

\begin{Theorem} (Lefschetz principle)
\label{T:lefschetz}
If $\phi$ is a sentence in $\Lr$ then:
\begin{enumerate}
\item If $\phi$ is true in some model of {\sf AC}$_p$ where $p \ge 0$ then $\phi$ is true in every model of {\sf AC}$_p$.
\item If $\phi$ is true in some model of {\sf AC}$_0$ then there exists $p_0 \in \N$ such that $\phi$ is true in any model of ${\sf AC}_p$ for $p > p_0$.
\end{enumerate}
\end{Theorem}

\subsection{Ideals in tensor products of algebras}

Suppose that $k$ is a field and that $A, B$ are finitely generated $k$-algebras. Let $K$ be an ideal of $A \otimes_k B$ with generators $f_1,...,f_n$.  Fix a basis $\Xi= \{c_\lambda \mid \lambda\in \Lambda\}$ for $B$ over $k$.  We can write $f_i = \sum_m a_{i,m}\otimes b_{i,m}$ for elements $a_{i,m} \in A$ and $b_{i,m} \in \Xi$. Without loss of generality we can assume the elements $\{b_{i,m} \mid m\}$ are distinct for each fixed $i$, and under this assumption the ideal $J \subseteq A$ generated by $\{a_{i,m} \mid i,m\}$ is uniquely determined by $K$, as we see from the next lemma.

\begin{Lemma}
\label{L:idealsintensors}
$J$ is the smallest ideal of $A$ such that $K \subseteq J \otimes_k B$.
\end{Lemma}

\begin{proof}
Certainly $K \subseteq J\otimes_k B$ and so it suffices to take an ideal $L \subseteq A$ such that $K \subseteq L \otimes_k B$ and show that $J \subseteq L$. Observe that $A\otimes_k B$ is a free $A$-module with basis $\{1\otimes c_\lambda \mid \lambda\in \Lambda\}$. Furthermore $L \otimes_k B = \bigoplus_{\lambda\in \Lambda} L \otimes c_\lambda$ and so if $f_i \in L \otimes_k B$ then $a_{i,m} \in L$ for all $m$ appearing in the expression $f_i = \sum_m a_{i,m} \otimes b_{i,m}$. It follows that $J \subseteq L$ as required.
\end{proof}

\section{Smoothness of centralisers and Gr{\"o}bner bases}
\label{S:smoothcentralisers}

\subsection{Bounded polynomials and Gr\"obner bases}
\label{ss:boundedpolysandGrobnerbases}

Throughout this section we fix $n \in \N$ and $k$ denotes an algebraically closed field.  We will want to quantify over all $k$-algebras of bounded presentation, equipped with the structure of a Hopf algebra of bounded presentation.  Here ``bounded'' means that the lengths and degrees of the polynomial expressions which appear in the defining ideal of the underlying affine algebra, together with the comultiplication, antipode and counit are bounded.  To do so, we need to formulate statements to say that the Hopf algebra axioms are satisfied. Our main tool to this end will be to quantify over all Gr\"obner bases of bounded degree. We refer to \cite[Ch.~15]{Eis95} for a hearty introduction to Gr\"obner bases, but for our purposes we collect a simplified version here. 

The basic principle is to provide a process for reduction of elements of $S:=k[x_1,\dots x_n]$ by elements of an ideal, which will terminate in a finite number of steps. Hence one wants to know when the size of an expression is reduced by an operation, and for this one first needs to choose a total order on monomials. This order needs to be \emph{admissible} in the sense that $m_1>m_2$ if and only if $m_1m_3>m_2m_3>m_2$ for any monomials $m_1,m_2,m_3$ such that $m_3 \ne 1$. 

We will demand of the order that for any monomial $m\in S$, there are only finitely many $m'<m$.  For instance, we may use the \emph{homogeneous} (or \emph{graded}) \emph{lexicographic ordering}, in which \begin{eqnarray*}m:=x_1^{a_1}\dots x_n^{a_n}>m':=x_1^{b_1}\dots x_n^{b_n}\text{ if and only if }\deg m>\deg m'\\\text{ or if }\deg m=\deg m' \text{ then }a_i>b_i\text{ for the first index $i$ with }a_i\neq b_i.\end{eqnarray*} Thus the set of monomials is isomorphic to $\N$ as a totally ordered set. We define $\m_r$ to be the \emph{$r$th monomial in $S$}; observe that $\m_{\ell_d}$ is then the top monomial of total degree $d$.  If $m$ is a monomial then we define $k^*\cdot m$ to be $\{\lambda m\,|\,\lambda\in k^*\}$ and we call elements of this form {\em terms}; every polynomial can be written uniquely as a sum of terms.
We extend $>$ to terms by defining $\lambda m_i> \mu m_j$ if $i> j$ and $0\neq \lambda, \mu\in k$, and we define the \emph{initial term} $\In(f)$ to be the greatest term appearing in $f$ with respect to $>$ (taking $\In(0)= 0$). For an ideal $I\subseteq S$, we define $\In(I)$ to be the ideal generated by the elements $\In(f)$ for all $f\in I$.  If $g$ and $h$ are terms then there is a unique monomial $m$ such that $m$ divides $g$ and $h$, and any other monomial dividing $g$ and $h$ also divides $m$: we define $\gcd(g,h)= m$.

\begin{Definition}\label{grobdef} A \emph{Gr\"obner basis} with respect to $>$ is an ordered list of elements $(g_1,\dots, g_t)\in S^t$ for some $t$ such that if $I$ is the ideal of $S$ generated by $g_1,\dots, g_t$, then $\In(g_1),\dots,\In(g_t)$ generate $\In(I)$.\footnote{In contrast to \cite{BW93}, but consistently with \cite{Eis95} we allow elements of Gr\"obner bases to be zero.}  Note that the $g_i$ need not be distinct and that although we work with ordered lists, the property of being a Gr\"obner basis does not depend on the ordering of the $g_i$. \end{Definition}

Fix $d\in \N$. We wish to view a polynomial in $S=k[x_1,\dots,x_n]$ as a finite list of its coefficients, where we will ultimately be quantifying over all possible lists of those coefficients. We define the degree $\deg(f)$ of a polynomial $f\in S$ to be the total degree of $\In(f)$. According to our chosen (homogeneous) monomial order, $\deg(f)$ is the highest total degree of any term in $f$. Let $\mathcal{X}_d\subset S$ be the set of monomials of degree at most $d$, and let $\ell_d:=|\mathcal{X}_d|$. By homogeneity of the monomial order again, this means $\mathcal{X}_d=\{\m_1,\dots,\m_{\ell_d}\}$. Furthermore, we may identify the set $S_d$ of polynomials of degree at most $d$ with the Cartesian product $k^{\ell_d}$: the polynomial $\sum_{i=1}^{\ell_d} \lambda_i \m_i$  corresponds to $(\lambda_1, \lambda_2,\dots,\lambda_{\ell_d}) \in k^{\ell_d}$.

\begin{Definition}
Let $\mathcal{S}=R[x_1,\dots,x_n]$ be a polynomial ring over a ring $R$. Let $\mathcal{S}_d$ denote the polynomials in $\mathcal{S}$ of degree at most $d$. We say that an ordered list $\B$ of polynomials in $R$ is \emph{$d$-bounded} if $|\B|=\ell_d$ and $\B$ consists of elements of $\mathcal{S}_d$. 

If $R=k$ and $\B$ is also a Gr\"obner basis, we say $\B$ is a $\emph{$d$-bounded Gr\"obner basis}$.
\end{Definition}

\noindent We identify the set of $d$-bounded lists of elements of $S$ with $S_d^{\ell_d} = k^{\ell_d^2}$. Observe the following:

\begin{Remarks}(i) Any Gr\"obner basis of length greater than $\ell_d$ consisting of polynomials of degree at most $d$ can be reduced to a Gr\"obner basis of cardinality at most $\ell_d$. For if there are at least $\ell_d+1$ elements then two, $f$ and $g$ say, must have the same leading monomial. So for some $\lambda\in k$, $g-\lambda f$ has a lower leading monomial and replacing $g$ by $g-\lambda f$ we still have a Gr\"obner basis, directly from Definition \ref{grobdef}. Inductively we may assume $g$ is zero, thus it can be removed to produce a smaller Gr\"obner basis.

(ii) Conversely, any finite list of polynomials (resp.~Gr\"obner basis) can be embedded into a $d$-bounded list of polynomials (resp.~$d$-bounded Gr\"obner basis) for some $d$ by appending an appropriate number of zero polynomials to the end of the list.
\label{grobrem}\end{Remarks}

\begin{Lemma}\label{initial} Let $d\in \N$ and let $1\leq e\leq \ell_d$. Then there is a first-order formula $\phi_{e,d}$ in the language $\Lr$ of rings with $\ell_d$ free variables such that for any polynomial $f\in S$ of degree at most $d$, \[\phi_{e,d}(f)\text{ holds }\iff \In(f) \in k^*\cdot \m_e. \] 
\end{Lemma}
\begin{proof} 
After identifying the set $S_d$ with the space $k^{\ell_d}$, so that the polynomial $f = \sum_{i=1}^{\ell_d} \lambda_i \m_i$ identifies with $(\lambda_1,...,\lambda_{\ell_d})$, then the required formula is
\[ (\lambda_e \neq 0) \wedge (\lambda_{e+1} = 0) \wedge (\lambda_{e+2} = 0) \wedge\cdots \wedge (\lambda_{\ell_d} = 0).\qedhere\]
\end{proof}

Given a $d$-bounded list of polynomials, we need to check with a first-order formula that it forms a Gr\"obner basis. For this, we appeal to Buchberger's criterion \cite[Theorem 15.8]{Eis95}, which we reproduce here.

Let $c$ be any integer and let $\B = (g_1,...,g_c) \in S^c$. For each pair of indices $1 \leq i,j \leq c$, 
we define $$m_{ij}=\In(g_i)/\gcd(\In(g_i),\In(g_j))\in S.$$ Then it follows from the division algorithm \cite[Prop.~15.6]{Eis95} that there exist $f_u^{(ij)}\in S$ with $\In(m_{ji}g_i-m_{ij}g_j)\geq\In(f_u^{(ij)}g_u)$ for each $1\leq u\leq c$, and remainders $h_{ij} \in S$, none of whose terms is in $(\In(g_1),\dots,\In(g_c))$, such that
\begin{eqnarray}
\label{e:Buch}
m_{ji}g_i-m_{ij}g_j=\left(\sum f_u^{(ij)}g_u\right)+h_{ij}.
\end{eqnarray}

We call an expression \eqref{e:Buch} a {\it standard expression} for $m_{ji}g_i-m_{ij}g_j$.

\begin{Theorem}
[Buchberger's Criterion]\label{buch} The set $\B$ is a Gr\"obner basis if and only if there exist standard expressions \eqref{e:Buch} such that $h_{ij} = 0$ for all $1\leq i,j\leq c$.
\end{Theorem}

\begin{Lemma}
\label{L:no_common_variables}
 If $g_i$ and $g_j$ have no variables in common, there is a standard expression for $g_i$ and $g_j$ such that $h_{ij}= 0$.
\end{Lemma}

\begin{proof}
Since $m_{ij}=\In(g_i)$ and $m_{ji}=\In(g_j)$, (\ref{e:Buch}) becomes \[m_{ji}g_i-m_{ij}g_j= \underbrace{-(g_j-\In(g_j))}_{f_i^{(ij)}}g_i+\underbrace{(g_i-\In(g_i))}_{f_j^{(ij)}}g_j.\] Write $g_i=\In(g_i)+\widetilde{\In(g_i)}+g_i'$ with $\widetilde{\In(g_i)}$ the initial term of $g_i- \In(g_i)$, and likewise for $g_j$. Then $m_{ji}g_i-m_{ij}g_j=\In(g_j)\widetilde{\In(g_i)}+ \In(g_j)g_i' -\In(g_i)\widetilde{\In(g_j)}-\In(g_i)g_j'$ has as initial term whichever is the larger of $-\In(g_j)\widetilde{\In(g_i)}$ and $\In(g_i)\widetilde{\In(g_j)}$ (these two terms cannot cancel each other as $g_i$ and $g_j$ have no variables in common).  But the initial terms of $f_i^{(ij)}$ and $f_j^{(ij)}$ are $-\widetilde{\In(g_j)}\In(g_i)$ and $\widetilde{\In(g_i)}\In(g_j)$, respectively, so $\In(m_{ji}g_i-m_{ij}g_j)\geq \In(f_u^{(ij)}g_u)$ for $u= i$ and $u= j$.  Hence we get $h_{ij}=0$, as required. \end{proof}

\begin{Lemma}Let $d\in \N$. Then there is a first-order formula $\beta_d$ in the language of rings with $\ell_d^2$ free variables such that if $\B$ is a $d$-bounded list of elements of $S$, then \[\beta_d(\B)\text{ holds}\iff \B\text{ is a Gr\"obner basis}.\]\label{grobtest}\end{Lemma}
\begin{proof}
Suppose $\B= (g_1,\dots,g_{\ell_d}) \in S_d^{\ell_d}$. We will produce a first-order formula which detects whether there exist expressions \eqref{e:Buch} for each pair $(g_i,g_j)$, with $h_{i,j} = 0$.  Suppose $\In(g_i)\in k^*\cdot \m_{a}$ and $\In(g_j)\in k^*\cdot \m_{b}$ and that there is a formula $\chi_{a,b}$ such that $\chi_{a,b}(g_i,g_j)$ is true if and only if there exist expressions \eqref{e:Buch} such that $h_{ij}=0$.  Then using Lemma \ref{initial} we set $\beta_d(\B)$ to be the formula
\[\bigwedge_{1,\leq i,j\leq \ell_d}\left(\bigvee_{1\leq a,b\leq \ell_d}\left(\chi_{a,b}(g_i,g_j)\wedge \phi_{a,d}(g_i)\wedge\phi_{b,d}(g_j)\right)\right),\] and we see that $\beta_d(\B)$ is true if and only if $\B$ satisfies the necessary and sufficient criterion of Buchberger's Criterion to deduce that $\B$ is a Gr{\"o}bner basis . 

Thus we have reduced the problem, without loss of generality, to showing the existence of $\chi_{a,b}(g_1,g_2)$. For fixed $a$ and $b$, $\m_{e'}:=\gcd(\m_{a},\m_{b})$ is also fixed, depending just on the bijection between $\N$ and the monomials in $S$, hence so are the monomials $m_{12}$ and $m_{21}$. Now, the highest monomial appearing in the left-hand side of \eqref{e:Buch} is at most the $d'$th, where $d'$ is given by $\m_{d'+1}=(\m_{a}\m_{b}/\m_{e'})$. Suppose $\In(m_{ji}g_i-m_{ij}g_j)=\m_e$. Then there is a finite set of pairs $P=\{(g_{a_b},\m_{a_b})\}_{1\leq b\leq p}$ such that $\In(g_{a_b}\m_{a_b})\leq \m_e$. Hence, setting $\chi_{e,a,b}(g_1,g_2)$ to be the formula
\[(\exists \lambda_{1})\dots(\exists \lambda_p)(m_{21}g_1-m_{12}g_2-\sum_{1\leq b\leq p}\lambda_{b}g_{a_b}\m_{a_b}=0),\] we see that $\chi_{e,a,b}(g_1,g_2)$ holds if and only if there is an expression of the form \eqref{e:Buch} for $g_1$ and $g_2$ with $h_{i,j} = 0$ (given $\In(m_{ji}g_i-m_{ij}g_j)=\m_e$). Lastly, let $\chi_{a,b}(g_1,g_2)$ be the formula
\[\bigvee_{e=1}^{d'}\left(\phi_{e,d}(m_{2,1}g_1-m_{1,2}g_2)\wedge \chi_{e,a,b}(g_1,g_2)\right);\]
this will do.\end{proof}

Another important thing we need to be able to encode with a first-order statement is the dimension $\dim (I)=\dim(\Spec_k(S/I))$ of the scheme determined by an ideal $I\subseteq S= k[x_1,...,x_n]$. If $I = ( g_1,...,g_{\ell_d})$ then in general it is not easy to read off $\dim I$ from the elements $\{g_1,...,g_{\ell_d}\}$.  However, when $\{g_1,...,g_{\ell_d}\}$ form a Gr\"{o}bner basis for $I$ there is a simple method: the dimension is the maximal size of a subset $X \subseteq \{x_1,...,x_n\}$ such that $\In(g_1),...,\In(g_n)$ depend only on the elements of $\{x_1,...,x_n\} \setminus X$ \cite[Defn.~9.22 \& Cor.~9.28]{BW93}. Using this fact along with Lemma \ref{initial}, we can determine dimension with a first-order formula.

\begin{Lemma} Let $d\in\N$ and $0\leq e\leq n$. Then there is a first-order formula $\delta_{e,d}$ in the language $\Lr$ of rings, with $\ell_d^2$ free variables, such that if $\B$ is any $d$-bounded Gr\"obner basis with $I=(\B)$, then  \[\delta_{e,d}(\B)\text{ holds}\iff \dim(I)=e.\]\label{groupdim}\end{Lemma}
\begin{proof}There is obviously a finite collection of lists of monomials which could play the role of initial terms of the elements of a $d$-bounded Gr\"obner basis which define an ideal of dimension $e$. More formally, there is a set $\mathcal{X}_e=\{X_j \mid j \in T\}$, where $T$ is some finite index set, and where each $X_j$ is a $d$-bounded list of monomials in $S$ satisfying
(i) there are distinct $i_1,\dots,i_e\in \{1,\ldots, ,n\}$ such that each $m\in X_j$ does not involve $x_{i_1},\dots, x_{i_e}$; (ii) for any distinct $i_1,\dots,i_{e+1}\in \{1,\ldots, n\}$, there is $m\in X_j$ depending on $x_{i_k}$ for some $1\leq k\leq e+1$. For convenience we assume that the $X_j$ are ordered sets and identify the monomials with their ordinal via the bijection of monomials of $S$ with $\N$. Then we may set $\delta_{e,d}(\B)$ to be the formula \[\bigvee_{X_j=(a_{1},\dots,a_{\ell_d})\in \mathcal{X}_e} (\phi_{a_{1},d}(g_1)\wedge\phi_{a_{2},d}(g_2)\wedge\dots\wedge\phi_{a_{\ell_d},d}(g_{\ell_d})).\qedhere\]\end{proof}

The next lemma uses the ideal membership algorithm for Gr\"obner bases to write a first-order formula whose truth determines whether an element is in an ideal. If $\B$ is a $d$-bounded Gr\"{o}bner basis and $f \in S_d$ then we may identify $(\B, f)$ with an element of $k^{{\ell_d}^2 + \ell_d}$ in the usual manner.

\begin{Lemma}\label{idealmem}
Let $d\in \N$. Then  there is a first-order formula $\iota_d$ in $\Lr$ with $\ell_d^2 + \ell_d$ free variables, such that for any $f\in S_d$ and $d$-bounded Gr\"obner basis $\B$ with $I:=(\B)$, then 
\[\iota_d(\B,f)\text{ holds} \iff f\in I.\]\end{Lemma}

\begin{proof}Let $(g_1,\ldots, g_{l_d})$ be a $d$-bounded Gr\"obner basis and let $f\in S_d$.  Since the elements of $\B$ have bounded total degree $d$, and $<$ is a homogeneous order, there are only finitely many monomials $m$ such that $\In(m g_i)\leq \In(f)$ for some $1\leq i\leq l_d$, where this number depends only on $d$. Let $\m_{d'}$ be the greatest such monomial. Thus we set $\iota_d(\B,f)$ to be the formula
\[(\exists\lambda_{i,j})_{1\leq i\leq d',1\leq j\leq d} \ \ f=g_1 \left(\sum_{i=0}^{d'}\lambda_{i,1}\m_i\right)+g_2 \left(\sum_{i=0}^{d'}\lambda_{i,2}\m_i\right)+\dots+g_d \left(\sum_{i=0}^{d'}\lambda_{i,d}\m_i\right).\tag{$\dagger$}\]

\noindent We claim that $\iota_d(\B,f)$ is true if and only if $f\in I$. This follows by induction on $e$ where $\In(f)=\m_e$: since $\B$ is a Gr\"obner basis, by \cite[5.35(vii)]{BW93}, $f$ is \emph{top-reducible} by some $g_i$ or is not in $I$. In the former case, this means that there is a term $m$ such that $\In(f-g_im)<\In(f)$. By the inductive hypothesis, $\iota_d(\B,f-g_im)$ is true whenever $f-g_im\in I$, which is the case if and only if $f\in I$. If $f-g_im\in I$ this says that there is an expression of the form $(\dagger)$ with $f$ replaced by $f-g_im$; moving $g_im$ to the other side of the equation, this says that there is also one for $f$.\end{proof}

\subsection{Complexity of schemes and their morphisms}
\label{subsec:complexity}
We recall some terminology, now reasonably common in the literature, to describe the boundedness of affine schemes; see \cite[Defn.~4.1]{Sch00}) for example.  It is closely related to the notion of $d$-boundedness above.

\begin{Definition}
\begin{itemize}
\setlength{\itemsep}{4pt}
\item[(a)] We say that an ideal $I$ of $S= R[x_1,\ldots, x_n]$ {\em has complexity at most $d$} if $I$ can be generated by polynomials of degree at most $d$; in this case we also say that the affine $R$-scheme $X$ corresponding to $I$ has complexity at most $d$.
\item[(b)] Let $X$ and $X'$ be the affine $R$-schemes determined by ideals $I$ of $S= R[x_1,\ldots, x_n]$ and $I'$ of $S'= R[x'_1,\ldots, x'_{n'}]$.  We say that a regular map of schemes $f\colon X\to X'$ {\em has complexity at most $d$} if the comorphism $f^*$ can be represented by an $R$-algebra homomorphism from $S'$ to $S$ involving polynomials of degree at most $d$. Write $\Hom_{R\alg}(S',S)_d$ for the set of all homomorphisms of degree at most $d$. 

\end{itemize}
\end{Definition}

Note that our definition differs slightly from that in \cite{Sch00}: we regard $n$ as fixed but we don't require that $n\leq d$, as this is not necessary for our purposes.  By an \emph{affine embedding} of an $R$-scheme $X$, we mean an embedding of $X$ in $\A^n_R$ for some $n\in \N$.  Below when we speak of the complexity of an $R$-scheme, we mean it to be taken with respect to a fixed affine embedding and likewise for morphisms of $R$-schemes; we will not mention the affine embedding explicitly unless it is necessary.  If $G$ is an algebraic $R$-group then we pick an affine embedding arising from a GroHo quadruple in the sense of Definition~\ref{grpcomplex} below.  If we are given affine embeddings of $R$-schemes $X_1$ and $X_2$ then we take our affine embedding of $X_1\times X_2$ to be the product embedding.

\begin{Remarks}\label{complexityremarks} (i)
\label{rem:boundedness}
If $I$ has complexity at most $d$ then a generating set of polynomials of degree at most $d$ can be transformed into a $d$-bounded list as per Remarks \ref{grobrem}(i) and (ii). This allows us to apply the results from Section~\ref{ss:boundedpolysandGrobnerbases} replacing hypotheses involving boundedness with hypotheses involving bounded complexity.

(ii)\label{rem:compn}
 Let $X$, $X'$ and $X''$ be affine schemes corresponding to ideals $I$ of $S$, $I'$ of $S'$ and $I''$ of $S''$, respectively.  Let $f\colon X\to X'$ and $g\colon X'\to X''$ be maps of complexity at most $r$ and $s$, respectively.  It follows immediately from the definitions that $g\circ f$ has complexity at most $rs$.  Similarly, if $Y'$ is a closed subscheme of $X'$ and $Y'$ has complexity at most $d$ then $f^{-1}(Y')$ has complexity at most $dr$---note that this bound does not depend on the complexity of $X$ and $X'$.

(iii) \label{rem:intersection}
 Suppose $\{X_i\,|\,i\in I\}$ is a family of affine $R$-schemes given by ideals $I_i$ of $S$.  It is immediate that if each $X_i$ has complexity at most $d$ then $\bigcap_{i\in I} X_i$ also has complexity at most $d$.

(iv) \label{rem:base_extn}
 The notion of complexity behaves well under base extension in the following sense.  Suppose $X$ is an affine $R$-scheme with a given embedding in some affine space ${\mathbb A}^n_R$.  Let $k$ be an $R$-field.  Then base extension gives an affine $k$-scheme $X_k$ with an embedding in ${\mathbb A}^n_k$.  If $f_1,\ldots, f_t\in R[x_1,\ldots, x_n]$ generate the vanishing ideal of $X$ in ${\mathbb A}^n_R$ then the images of the $f_i$ in $k[x_1,\ldots, x_n]$ generate the vanishing ideal of $X_k$ in ${\mathbb A}^n_k$.  Hence if $X$ has complexity at most $d$ then $X_k$ also has complexity at most $d$.  The analogous result holds for morphisms.

 (v) Suppose $X$ and $Y$ are both affine $R$-schemes, say 
 \[k[X]=k[x_1,\dots,x_r]/I\quad \text{and}\quad k[Y]=k[y_1,\dots, y_s]/J\] with \[I=(f_1,\dots, f_t) \quad \text{and}\quad J=(g_1,\dots,g_u).\] 
 Then 
 \[k[X\times Y]\cong k[X]\otimes k[Y]\cong k[x_1,\dots,x_r,y_1,\dots y_s]/K\] where $K$ is generated by the concatenation of the $f_i$ and $g_i$ (after extending the domain of $f_i$ and $g_i$ to be trivial functions of the $y_j$ and $x_j$ respectively). Then one sees that if the complexity of $X$ is $d_1$ and of $Y$ is $d_2$, we get a presentation of $k[X\times Y]$ which shows its complexity is at most $\max\{d_1,d_2\}$. In particular, arguing by induction shows that if $S/I$  has complexity at most $d$, then $(S/I)^{\otimes r}$ has complexity at most $d$.
\end{Remarks}
\begin{Definition}
\label{defn:action_complexity}
 Let $G$ be an algebraic $R$-group and let $X$ be an affine $G$-scheme with action map $\alpha:G\times X\to X$.  We say that \emph{the action of $G$ on $X$ has complexity at most $d$} if $\alpha$ does.
\end{Definition}

\begin{Definition}
\label{grpcomplex}
Let $H:=(\B,\Delta,\sigma,\epsilon)$ be a quadruple with $(\B)=I\unlhd S$; and $\Delta\colon S\to S^{\otimes 2}$, $\sigma\colon S\to S$ and $\epsilon\colon S\to k$ being $k$-algebra homomorphisms. Then we say $H$ is a \emph{Hopf-quadruple} if $S/I$ equipped with the maps $\Delta$, $\sigma$ and $\epsilon$ forms a Hopf algebra. We say that a Hopf-quadruple is of {\em complexity at most $d$} if $\B$ consists of polynomials of degree at most $d$ and the complexity of $\Delta$, $\sigma$ and $\epsilon$ are at most $d$. Dually, if an affine algebraic $k$-group $G$ is descibed by the data $(\Spec(S/(\B)),\Delta^*,\sigma^*,\epsilon^*)$, where $(\B,\Delta,\sigma,\epsilon)$ is a Hopf quadruple of complexity at most $d$, then we say that $G$ is an \emph{algebraic group of complexity at most $d$}.

In case $R=k$ is a field, we call such a quadruple a \emph{GroHo-quadruple} if $\B$ happens to be a Gr\"obner basis.
\end{Definition}

\begin{Lemma}
\label{lem:factors}
 Let $d,r\in\N$. Then there is a first-order formula $\zeta_{d,r}$ in $\Lr$ with $n\ell_{d,r}+\ell_d^2$ free variables such that if $\Lambda\colon S\to S^{\otimes r}$ is a homomomorphism of complexity at most $d$ and $\B=\{f_1,\dots, f_{\ell_d}\}$ is any $d$-bounded Gr\"obner basis for an ideal $I$ of $S$,
\[\zeta_{d,r}(\B,\Lambda) \text{ holds } \iff \Lambda \text{ factors to a homomorphism } S/I\to (S/I)^{\otimes r}.\]
\end{Lemma}

\begin{proof} Recall $S=k[x_1,\dots,x_n]$ and $I=(\B)$. With reference to Remark \ref{complexityremarks}(v), we may take a presentation of $(S/I)^{\otimes r}$ as 
      \[(S/I)^{\otimes r}\cong k[x_{i,j}]/K\] where $1\leq i\leq n$, $1\leq j\leq r$ and $K$ is the generated by disjoint union  $\B_r:=\{f_{i,j}\}$ where $f_{i,j}$ acts on $x_{l,m}$ as zero if $j\neq m$ and otherwise acts as $f_i$ acts on the $x_l$. 
 
 To deploy our Gr\"obner basis formulae from earlier, we need to specify a monomial order on the $x_{i,j}$'s. Define first an order on the $(i,j)$ for $1\leq i\leq n$ and $1\leq j\leq r$ by $(i,j)< (i',j')$ if $j< j'$ or $j= j'$ and $i< i'$.  Now define an order on monomials by $m:= \prod_{i,j} x_{i,j}^{a_{i,j}}< m':= \prod_{i,j} x_{i,j}^{b_{i,j}}$ if and only if $\deg m< \deg m'$ or if $\deg m= \deg m'$ and $a_{i,j}< b_{i,j}$ for the greatest pair $(i,j)$ such that $a_{i,j}\neq b_{i,j}$.  We claim that $\B_r$ as above is a Gr\"obner basis for the homogeneous lexicographic monomial order on the monomials in $x_{i,j}$. Since our order extends the monomial orders on the subalgebras $k[x_{1,i},...,x_{n,i}]$ for any fixed $i$, we see that Buchburger's criterion (Theorem~\ref{buch}) holds for all pairs $(f_{i,j},f_{i',j})$; furthermore if $j\neq j'$ then $f_{i,j}$ and $f_{i',j'}$ have no variables in common, so Buchburger's criterion holds for $(f_{i,j},f_{'i,j'})$ by Lemma~\ref{L:no_common_variables}. This proves the claim.

Now since $\Lambda$ is has complexity at most $d$ and $\B_r$ is $d$-bounded, we may appeal to Lemma \ref{idealmem} to get first-order formulas $\iota_{d,r}$ such that $\iota_{d,r}(\B_r,\varphi_r(\Lambda(x_i)))$ holds if and only if $\varphi_r(\Lambda(x_i))\in J_r$. Hence we set $\zeta_{d,r}$ to be the formula
\[\bigwedge_{i=1}^n\iota_{d,r}(\B_r , \varphi_r(\Lambda(x_i))).\qedhere\]
\end{proof}

Recall the axioms of a Hopf algebra, listed in \eqref{assoc}--\eqref{inv}.

\begin{Lemma}\label{hopf}Let $d\in\N$.  There is a formula $\eta_d\in\Lr$ with $\ell_d^2+n(\ell_{d,2}+\ell_d+1)$ free variables such that if $\B$ is any $d$-bounded Gr\"obner basis, with $I=(\B)$ and $\Delta\colon S\to S^{\otimes 2}$, $\sigma\colon S\to S$ and $\epsilon\colon S\to k$ any $d$-bounded homomorphisms, then \[\eta_d(\B,\Delta,\sigma,\epsilon)\text{ holds}\iff(S/I,\Delta,\sigma,\epsilon)\text{ is a Hopf algebra.}\]\end{Lemma}
\begin{proof}Suppose $\Delta$, $\sigma$ and $\epsilon$ factor as $S/I\to S/I^{\otimes r}$. We must find formulas $\eta_d^{(1)}$, (resp.~$\eta_d^{(2)},\eta_d^{(3)}$) which hold if and only if (\ref{assoc}), (resp.~(\ref{ident}), (\ref{inv})) are satisfied. Since the constructions are almost identical for each formula, we give the details for $\eta_d^{(1)}$. To see that (\ref{assoc}) holds, it clearly suffices to check that $(\Delta\otimes\id)\circ\Delta(x_i+I)-(\id\otimes\,\Delta)\circ \Delta(x_i+I)=0\in(S/I)^{\otimes 3}\cong S^{\otimes 3}/J_3$ for each $1\leq i\leq n$, where $J_3=(\B_3)$ as above. This amounts to checking that $(\Delta\otimes\id)\circ\Delta(x_i)-(\id\otimes\,\Delta)\circ \Delta(x_i)\in J$. Since $\Delta$ is $d$-bounded, $\varphi_2(\Delta(x_i))$ is a $d$-bounded polynomial in $S^{\otimes 2}$; similarly $f_i:=\varphi_3((\Delta\otimes\id)\circ(\Delta(x_i)))$ is a $d^2$-bounded polynomial in $S^{\otimes 3}$. We have also that $\B_3$ is $d$-bounded. Thus we may set $\eta_d^{(1)}(\B,\Delta,\sigma,\epsilon)$ to be the formula
\[\bigwedge_{i=1}^n\iota_{d^2}(\B_3,f_i).\]
Finally, we set $\eta_d(\B,\Delta,\sigma,\epsilon)$ to be the formula

\[\zeta_{d,2}(\B, \Delta)\wedge\zeta_{d,1}(\B, \sigma)\wedge\zeta_{d,0}(\B, \epsilon)\wedge\eta_d^{(1)}(\B,\Delta,\sigma,\epsilon)\wedge \eta_d^{(2)}(\B,\Delta,\sigma,\epsilon)\wedge\eta_d^{(3)}(\B,\Delta,\sigma,\epsilon),\]
where the $\zeta_{d,r}$ are as in Lemma~\ref{lem:factors}.
\end{proof}

\subsection{Generic smoothness of algebraic groups of bounded complexity}
Here we invoke the Lefschetz principle and our first order formulas to show that algebraic groups of bounded complexity are generically smooth. We use the fact that when $k$ is a field, an algebraic $k$-group  $G$ is smooth if and only if $\dim(G)=\dim(\Lie(G))$ \cite[I.7.17]{Jan03}.

As a $k$-vector space, the Lie algebra $\Lie(G)$ is the tangent space $T_{\epsilon^*}(G)$, where $\epsilon^*$ is the identity element of $G$. Thus its dimension is the nullity of the $\ell_d\times n$ matrix $\mathscr{J}$ where $\mathscr{J}_{kl}=\epsilon(\partial f_k/\partial x_l)$.

\begin{Lemma}Let $d\in \N$, and $0\leq e\leq d$. There is a first-order formula $\tau_{e,d}$ in $\Lr$ with $\ell_d^2$ free variables such that for any GroHo quadruple $(\B,\Delta,\sigma,\epsilon)$ of complexity at most $d$ that describes the algebraic $k$-group  $G$ we have
\[\tau_{e,d}(\B)\text{ holds}\iff \dim\Lie(G)=e.\]
\label{liedim}\end{Lemma}
\begin{proof}As we identify each $f_i\in\B$ with the set of its $\ell_d$ coefficients $\lambda_{ij}$, partial differentiation by $\partial/\partial x_i$ gives a map $k^{\ell_d}\to k^{\ell_d}$. Composing with $\epsilon$ is then a map $k^{\ell_d}\to k$. Hence each $\mathscr{J}_{kl}$ is a fixed linear combination of the $\lambda_{ij}$'s. The statement that the nullity of $\mathscr{J}$ is $e$ is equivalent to the statement that there are $e$ linearly independent vectors $v_1,\dots, v_e\in k^{\ell_d}$ satisfying $\mathscr{J}\cdot v_i=0$ and any $e+1$ linearly independent set of $v_1,\dots,v_{e+1}\in k^{\ell_d}$ contains $v\in\langle v_1,\dots,v_{e+1}\rangle$ such that $\mathscr{J}\cdot v\neq 0$. This statement can be given as a formula in $\Lr$ 
in an obvious way (see \cite[Example~2.1(i)]{MSTT18} for example).
\end{proof}

\begin{Lemma}\label{gsmooth}Let $d\in\N$. Then there is a first-order formula $\theta_d$ with $\ell_d^2$ free variables such that for any GroHo quadruple $H:=(\B,\Delta,\sigma,\epsilon)$ of complexity at most $d$, \[\theta_d(\B)\text{ holds}\iff H\text{ describes a smooth $k$-group}.\]\end{Lemma}
\begin{proof}The $k$-group $G$ described by $H$ is a subscheme of $\Spec(S)\cong \A^n$, so $0\leq\dim G\leq n$. Then invoking Lemmas \ref{liedim} and \ref{groupdim} we may set $\theta_d(\B,\Delta,\sigma,\epsilon)$ to be the following formula:
\[\bigvee_{e=0}^n\left(\delta_{e,d}(\B)\wedge\tau_{e,d}(\B)\right).\qedhere\]\end{proof}

We wish to apply the above results to obtain statements about the set of algebraic groups of complexity at most $d$. This amounts to statements about the Hopf quadruples of complexity at most $d$. The following theorem of Dub\'e guarantees that any algebraic group of complexity at most $d$ can in fact be described by a GroHo quadruple of complexity at most $D$ where $D$ depends just on $d$ and $n$.

\begin{Theorem}[{\cite{Dub90}}]\label{dube} If $\B'\subset S$ is a set of polynomials of degree at most $d$, then $(\B')$ has a Gr\"obner basis $\B$ consisting of polynomials of degree at most
\[2\left(\frac{d^2}{2}+d\right)^{2^{n-1}}.\]\end{Theorem}

\begin{Theorem}
\label{gensmooth}
Let $d\in\N$. Then there is a prime $p_0=p_0(n, d)$ such that whenever $\Char (k) \geq p_0$, any algebraic $k$-group of complexity at most $d$ is smooth.\end{Theorem}

\begin{proof}By Theorem \ref{dube}, each affine algebraic group of complexity at most $d$ is described by a GroHo quadruple of complexity at most $D$ where $D$ depends just on $d$ and $n$; here $n$ has been fixed. So it suffices to prove that there is a $p_0(D)$ such that any GroHo quadruple of complexity at most $D$ over an algebraically closed field $k$ of characteristic $p\geq p_0$ describes a smooth algebraic $k$-group.

So let $H=(\B,\Delta,\sigma,\epsilon)$ be a GroHo quadruple of complexity at most $D$. Recall we identify $H$ with a string of $\ell_D^2+n(\ell_{D,2}+2\ell_D+1)$ coefficients in the field, which we write $(\lambda_i)_{i=1}^{\ell_D^2+n(\ell_{D,2}+2\ell_D+1)}$. Then invoking Lemmas \ref{grobtest}, \ref{hopf} and \ref{gsmooth}, the following formula $\Phi_D$ is a \emph{sentence} in $\Lr$, which is true if and only if all GroHo quadruples of complexity at most $D$ describe smooth algebraic groups:
\[(\forall\lambda_1)\cdots(\forall\lambda_{\ell_D^2+n(\ell_{D,2}+2\ell_D+1)})(\beta_D(\B)\wedge \eta_D(\B,\Delta,\sigma,\epsilon)\wedge\theta_D(\B)).\]
By Cartier's theorem \cite[I.7.17(2)]{Jan03}, $\Phi$ is true for all algebraically closed fields of characteristic $0$. Therefore the Lefschetz principle (Theorem \ref{T:lefschetz}) guarantees the existence of a prime $p_0$ such that the same is true for all algebraically closed fields of characteristic $p\geq p_0$.  The theorem follows.\end{proof}

\section{Normalisers and centralisers}
\label{S:normalandcent}
We now prove our main results Theorems \ref{thm:smoothcent} and ~\ref{thm:smoothnorm}.  For $k$ a field, smoothness is a geometric property, meaning that a $k$-group $G$ is smooth if and only if $G_{\bar k}$ is smooth.  In what follows, one may assume $k=\bar k$ for convenience.

We will work with non-affine schemes, so we need some preliminary material.  We use the definitions and terminology of \cite[Chapter I]{Jan03}.  Although we work mainly over $k$, we need to consider arbitrary $R$ and study the behaviour of some of the constructions below under base change to $k$.  Note that every $R$-scheme is locally free in the sense of \cite[I.1.15]{Jan03} if $R$ is a field.  We fix an $R$-scheme $X$ which is locally finitely presented and an open covering $\{X_i\,|\,i\in {\mathbb I}\}$ of $X$ such that each $X_i$ is affine and finitely presented.  For each $i$, we fix an affine embedding of $X_i$ in some ${\mathbb A}^{n_i}_R$.  This allows us to talk about the complexity of certain schemes and maps involving the $X_i$.

We also fix an algebraic $R$-group $G$ (not necessarily smooth) described by a Hopf quadruple $(S/I, \Delta, \sigma, \epsilon)$.  We fix an action $\alpha\colon G\times X\to X$ of $G$ on $X$.

\subsection{The functor \texorpdfstring{$\mathfrak{Mor}(\mbox{--},\mbox{--})$}{Mor}}

\label{sec:trans}

We review some basic constructions of algebraic geometry: see \cite[I.1.15,\ I.2.6]{Jan03}.
Recall that for $R$-schemes $Z$ and $W$, we get an $R$-functor \[\mathfrak{Mor}(Z,W)\colon \underline{R\mathrm{-Alg}}\to\underline{\mathrm{Set}}\] given by $\mathfrak{Mor}(Z,W)(Q)= {\rm Mor}(Z_Q,W_Q)$ for any $R$-algebra $Q$; if $\varphi\colon Q\to Q'$ is a homomorphism of $R$-algebras then $\mathfrak{Mor}(Z,W)(\varphi)\colon {\rm Mor}(Z_Q,W_Q)\to {\rm Mor}(Z_{Q'},W_{Q'})$ is given by base extension. 
Given another $R$-scheme $C$ and a morphism $\alpha\colon C\times Z\to W$ of $R$-schemes, we get a map of $R$-functors $\nu\colon C\to \mathfrak{Mor}(Z,W)$ such that for an $R$-algebra $Q$ and $c\in C(Q)$, $\nu$ maps $c$ to $\alpha(c, - )\in \mathrm{Mor}(Z_Q,W_Q)$.

Now let $W'$ be a closed subscheme of $W$.  If $Z$ is locally free as an $R$-scheme then by \cite[I.1.15]{Jan03} we may regard $\mathfrak{Mor}(Z,W')$ as a closed $R$-subfunctor of $\mathfrak{Mor}(Z,W)$.  So by \cite[I.1.15(3)]{Jan03} we obtain a closed subscheme $\nu^{-1}(\mathfrak{Mor}(Z,W'))$ of $C$.

\subsection{\texorpdfstring{$G$-complexity}{G-complexity}}
\label{sec:G-complexity}
Next we describe the precise boundedness condition that we need.  With the notation of the last subsection, let  $Y$ be a closed subscheme of $X$. Set $Y_i= Y\cap X_i$ for each $i$.  Recall $A:=R[G]\cong S/I$ for some polynomial algebra $S$ over $R$ and let $B_i= R[Y_i]$.  Let $\alpha_i\colon G\times Y_i\to X$ be the restriction of $\alpha$ to $Y_i$.  Then $\alpha_i^{-1}(Y)$ is a closed subscheme of $G\times Y_i$ \cite[I.1.12(2)]{Jan03}, so it corresponds to an ideal $K_i$ of $R[G\times Y_i]= A\otimes_k B_i$.  We can write $K_i= (f^{(i)}_1,\ldots, f^{(i)}_{t_i})$, where each $f^{(i)}_j$ has the form
\begin{equation}
\label{eqn:preimagegen}
 f^{(i)}_j= \sum_m a^{(i)}_{mj}\otimes b^{(i)}_{mj}\end{equation}
for some $a^{(i)}_{mj}\in A$ and $b^{(i)}_{mj}\in B_i$.

\begin{Definition}
\label{defn:Nbdd}
 Let $G$, $X$ and $Y$ be as above.  We say that {\em $Y$ has $G$-complexity at most $d$} if there exist $f^{(i)}_j$ as above such that each $a^{(i)}_{mj}$ has a representative $\widetilde{a}^{(i)}_{mj}$ in $S$ of degree at most $d$.
\end{Definition}

\noindent Note that we do not place any restrictions on the degrees of the $b^{(i)}_{mj}$ in the definition.

\begin{Remark}
\label{rem:loc_complexity}
 Let $G$, $X$ and $Y$ be as above.  The $G$-complexity condition in Definition~\ref{defn:Nbdd} can be hard to verify, but here is a useful special case.  Let $X$ be affine. Suppose $Y\subseteq X$ has complexity at most $e$ and the action $\alpha:G\times X\to X$ has complexity at most $e'$. Then $Y$ has $G$-complexity at most $ee'$. 
 
 To see this, let $k[Y]=k[X]/J$. Then the ideal of $k[G]\otimes k[Y]$ determining the $G$-complexity is $\alpha^*(J)$. If the polynomials in $J$ have degree at most $e'$ then applying the algebra homomorphism $\alpha^*$ to them leads to polynomials of degree at most $ee'$. 
\end{Remark}

Now let $Q$ be an $R$-algebra.  Extending scalars yields an algebraic $Q$-group $G_Q$ acting on a locally finitely presented $Q$-scheme $X_Q$ with an open covering by affine schemes $(X_i)_Q$.  If $Y$ is a closed subscheme of $X$ then $Y_Q$ is a closed subscheme of $X_Q$.  The various constructions in Section~\ref{sec:G-complexity} are well-behaved with respect to base extension, so we see that if $Y$ has $G$-complexity at most $d$ then $Y_Q$ has $G_Q$-complexity at most $d$.

\subsection{Normalisers}
\label{subsec:normalisers}
We start by recalling the scheme-theoretic definition of the normaliser (see \cite[I.2.6]{Jan03} for details).  In this subsection we assume that $G$, $X$, etc., are defined over the algebraically closed field $k$: this means that the local freeness condition used in Section~\ref{sec:trans} holds.  Let $Y$ be a closed subscheme of $X$.  Then the {\em normaliser of $Y$} (denoted $N_G(Y)$) is the $k$-subgroup functor of $G$ given by
$$ N_G(Y)(Q)= \{g\in G(Q)\,|\,\alpha(g,h)\in Y(Q')\ \mbox{for all $h\in Y(Q')$ and all } Q\mbox{-algebras } Q'\}. $$
We will need the explicit description of this subgroup afforded by the $k$-functor $\mathfrak{Mor}(\mbox{--},\mbox{--})$.  The action $\alpha\colon G\times X\to X$ gives rise to a map $\nu\colon G\to \mathfrak{Mor}(X,X)$ as described in Section~\ref{sec:trans}.  We have a map $\mathfrak{Mor}(X,X)\to \mathfrak{Mor}(Y,X)$ given by restriction, and we let $\gamma\colon G\to\mathfrak{Mor}(Y,X)$ be the composition with $\nu$.  Then $N_G(Y)= \gamma^{-1}(\mathfrak{Mor}(Y,Y))\cap i_G(\gamma^{-1}(\mathfrak{Mor}(Y,Y)))$, where $i_G\colon G\to G$ is the inversion map. As $\mathfrak{Mor}(Y,Y)$ is closed in $\mathfrak{Mor}(Y,X)$, it follows that $N_G(Y)$ is a closed subgroup functor of $G$, so $N_G(Y)$ is an algebraic $k$-group since $k$ is Noetherian.

\begin{Theorem}
\label{thm:normbdd}
 Let $d\in {\mathbb N}$ and let $G$, $X$ and $Y$ be as above.  Suppose $G$ has complexity at most $d$ and $Y$ has $G$-complexity at most $d$. Then $N_G(Y)$ is a closed subscheme of $G$ of complexity at most $d^2$.
\end{Theorem}

\begin{proof}
 We have $N_G(Y)= \gamma^{-1}(\mathfrak{Mor}(Y,Y))\cap i_G(\gamma^{-1}(\mathfrak{Mor}(Y,Y)))$.  By Remark~\ref{complexityremarks}(ii), it is enough to prove that $\gamma^{-1} (\mathfrak{Mor}(Y, Y))$ has complexity at most $d$ (since $i_G$ has complexity at most $d$).
 
Each $Y_i$ is a closed subscheme of $X_i$ \cite[I.1.13, Lemma]{Jan03}, and the $Y_i$ form an open cover of $Y$.  Let $\gamma_i$ be the composition $G\stackrel{\gamma}{\rightarrow}\mathfrak{Mor}(Y,X)\to \mathfrak{Mor}(Y_i,X)$, where the second map is given by restriction.  By \cite[I.1.15(2)]{Jan03}, $\gamma^{-1}(\mathfrak{Mor}(Y,Y))= \bigcap_{i\in {\mathbb I}} \gamma_i^{-1}(\mathfrak{Mor}(Y_i,Y))$.  Each $\gamma_i^{-1}(\mathfrak{Mor}(Y_i,Y))$ is a closed subscheme of $G$ by an argument similar to the one for $\gamma^{-1}(\mathfrak{Mor}(Y,Y))$.  In the proof of \cite[I.1.15(3)]{Jan03}, one considers a map $f$ from ${\rm Sp}_k(R)$ to $\mathfrak{Mor}(Y_i,X)$, where $R$ is an arbitrary $k$-algebra; one obtains a map $f'$ from ${\rm Sp}_k(R)\times Y_i$ to $X$.  We apply this construction when $R= A=k[G]$ and $f= \gamma_i$; then $f'$ is just the map $\alpha_i\colon G\times Y_i\to X$.  Let $K_i$ be the ideal of $k[G\times Y_i]= A\otimes_k B_i$ corresponding to $\alpha_i^{-1}(Y)$ as before, and let $K_i'$ be the ideal of $A$ corresponding to $\gamma_i^{-1}(\mathfrak{Mor}(Y_i,Y))$.  By the argument of {\em loc.\ cit.}, $K_i'$ is the smallest ideal of $A$ such that $K_i'\otimes B_i$ contains $K_i$.
 
 Fix $i\in {\mathbb I}$.  Choose $f^{(i)}_j$, $a^{(i)}_{mj}$, $\widetilde{a}^{(i)}_{mj}$ and $b^{(i)}_{mj}$ as in Definition~\ref{defn:Nbdd}.  Rewriting and expanding as necessary, we may assume each $b^{(i)}_{mj}$ is a member of a fixed $k$-basis of $B_i$; evidently this does not affect the bound on the degree of the $\widetilde{a}^{(i)}_{mj}$.  Thanks to Lemma~\ref{L:idealsintensors} we see that $K_i'$ is generated by the $a_{m j}^{(i)}$.
 
 Let $I_0$ be the ideal of $S$ corresponding to $\gamma^{-1}(\mathfrak{Mor}(Y,Y))$.  The argument above implies that $I_0$ is generated by the $\widetilde{a}^{(i)}_{mj}$ together with some generators for $I$.  But all of these elements have degree at most $d$, which shows that $\gamma^{-1}(\mathfrak{Mor}(Y,Y))$ has complexity at most $d$, as required. The result now follows.
\end{proof}

\begin{Corollary}
\label{cor:smoothnorm}
 Let $d, n\in {\mathbb N}$.  Then there is a prime $p_1= p_1(n,d)$ such that if:
 \begin{itemize}
\item $k$ is any field of characteristic $p\geq p_1$;
\item $G$ is any affine algebraic $k$-group of complexity at most $d$;
\item $X$ is any locally finitely presented $k$-scheme on which $G$ acts;
\item $Y$ is any closed subscheme of $X$ of $G$-complexity at most $d$;
\end{itemize} 
 then $N_G(Y)$ is a smooth closed subscheme of $G$.
\end{Corollary}

\begin{proof}
 This follows from Theorem~\ref{gensmooth}
 and Theorem~\ref{thm:normbdd}. 
\end{proof}

\subsection{\texorpdfstring{$(G, \Delta)$-complexity}{G,Delta-complexity}}
\label{sec:GD-complexity}
We need a slightly different boundedness condition for Theorem \ref{thm:centbdd}.  Here we assume that $X$ and $Y$ are defined over arbitrary $R$.  We also assume that $X$ is separated; this means that the diagonal $\Delta_X$ is closed in $X\times X$.  Set $Y_i= Y\cap X_i$ for each $i$ as before.  Let $A= k[G]$ and let $B_i= k[Y_i]$ for each $i$.

Define $\beta\colon G\times X\to X\times X$ by $\beta= \alpha\times {\rm pr}_2$, where ${\rm pr}_2\colon G\times X\to X$ is projection, and let $\beta_{Y}\colon G\times Y\to X\times X$ be the restriction of $\beta$.  Let $\beta_{Y_i}\colon G\times Y_i\to X\times X$ be the restriction of $\beta$.  Thanks to \cite[I.1.12(2)]{Jan03} we see that $\beta_{Y_i}^{-1}(\Delta_X)$ is a closed subscheme of $G\times Y_i$, so it corresponds to an ideal $K_i$ of $k[G\times Y_i]= A\otimes_k B_i$.  We can write $K_i= (f^{(i)}_1,\ldots, f^{(i)}_{t_i})$, and each $f^{(i)}_j$ has the form
\begin{equation}
\label{eqn:preimagegen2}
 f^{(i)}_j= \sum_m a^{(i)}_{mj}\otimes b^{(i)}_{mj}
\end{equation}
for some $a^{(i)}_{mj}\in A$ and some $b^{(i)}_{mj}\in B_i$.

\begin{Definition}
\label{defn:Cbdd}
 Let $G$, $X$ and $Y$ be as above.  We say that {\em $Y$ has $(G,\Delta)$-complexity at most $d$} if there exist $f^{(i)}_j$ as above such that each $a^{(i)}_{mj}$ has a representative in $S$ of degree at most $d$.
\end{Definition}

\begin{Lemma}
\label{lem:diagbdd}  
 Let $G$, $X$ and $Y$ be as above. Suppose $X$ has $(G,\Delta)$-complexity at most $d$.  Then $Y$ has $(G,\Delta)$-complexity at most $d$.
\end{Lemma}

\begin{proof}
 Fix $i\in I$.  Let $\beta_i\colon G\times X_i\to X\times X$ be the restriction of $\beta$.  By hypothesis, $\beta_i^{-1}(\Delta_X)$ is the closed subscheme of $G\times X_i$ defined by elements of the form $f^{(i)}_j= \sum_m a^{(i)}_{mj}\otimes b^{(i)}_{mj}$ for some $a^{(i)}_{mj}\in A$ and some $b^{(i)}_{mj}\in B_i$, where each $a^{(i)}_{mj}$ has a representative in $S$ of degree at most $d$.  Let $h^{(i)}_1,\ldots, h^{(i)}_{r_i}$ be generators for the ideal of $Y_i$.  Since $\beta_{Y_i}$ is the restriction of $\beta_i$ to $G\times Y_i$, we have $\beta_{Y_i}^{-1}(\Delta_X)= \beta_i^{-1}(\Delta_X)\cap (G\times Y_i)$, so the ideal of $\beta_{Y_i}^{-1}(\Delta_X)$ is generated by the $f^{(i)}_j$ for $1\leq j\leq n_i$ together with the elements $1\otimes h^{(i)}_j$ for $1\leq j\leq r_i$.  The result now follows.
\end{proof}

Now let $Q$ be an $R$-algebra.  We see that if $X$ has $(G, \Delta)$-complexity at most $d$ then $X_Q$ has $(G_Q, \Delta_Q)$-complexity at most $d$.

\subsection{Centralisers}\label{subsec:centralisers}
The argument is similar to the one for normalisers.  In this subsection we assume that $G$, $X$, etc., are defined over the algebraically closed field $k$: this means that the local freeness condition used in Section~\ref{sec:trans} holds.  We also assume that $X$ is separated.  We start by recalling the scheme-theoretic definition of the centraliser (see \cite[I.2.6]{Jan03} for details).  Let $Y$ be a closed subscheme of $X$.  Then the {\em centraliser of $Y$}, denoted $C_G(Y)$, is the $k$-subgroup functor of $G$ given by
$$ C_G(Y)(Q)= \{g\in G(Q)\,|\,\alpha(g,h)= h\ \mbox{for all $h\in Y(Q')$ and all } Q\mbox{-algebras } A'\}. $$
Again we require the $k$-functor $\mathfrak{Mor}(\mbox{--},\mbox{--})$.  The maps $\beta$ and $\beta_Y$ from Section~\ref{sec:G-complexity} give rise to maps $\delta\colon G\to \mathfrak{Mor}(X,X\times X)$ and $\delta_Y\colon G\to \mathfrak{Mor}(Y,X\times X)$, respectively, as described in Section~\ref{sec:trans}.  We have a map $\mathfrak{Mor}(X,X\times X)\to \mathfrak{Mor}(Y,X \times X)$ given by restriction, and it is easily seen that $\delta_Y$ is the composition $G\stackrel{\delta}{\to} \mathfrak{Mor}(X,X\times X)\to \mathfrak{Mor}(Y,X \times X)$.

Since $X$ is separated, $\Delta_X$ is closed in $X\times X$.  Then $C_G(Y)= \delta^{-1}(\mathfrak{Mor}(Y,\Delta_X))$, and this is a closed subgroup functor of $G$; in particular, $C_G(Y)$ is an algebraic $k$-group.

\begin{Theorem}
\label{thm:centbdd}
 Let $d\in {\mathbb N}$ and let $G$, $X$ and $Y$ be as above.  Suppose $G$ has complexity at most $d$ and $X$ has $(G,\Delta)$-complexity at most $d$. Then $C_G(Y)$ is a closed subscheme of $G$ of complexity at most $d$.
\end{Theorem}

\begin{proof}
Each $Y_i$ is a closed subscheme of $X_i$ \cite[I.1.13, Lemma]{Jan03}, and the $Y_i$ form an open cover of $Y$.  Let $\delta_i$ be the composition $G\stackrel{\gamma}{\rightarrow}\mathfrak{Mor}(Y,X\times X)\to \mathfrak{Mor}(Y_i,X\times X)$, where the second map is given by restriction.  We have $C_G(Y)= \delta^{-1}(\mathfrak{Mor}(Y,\Delta_X))$.  By \cite[I.1.15(2)]{Jan03}, $\delta^{-1}(\mathfrak{Mor}(Y,\Delta_X))= \bigcap_{i\in {\mathbb I}} \delta_i^{-1}(\mathfrak{Mor}(Y_i,\Delta_X))$.  Each $\delta_i^{-1}(\mathfrak{Mor}(Y_i,\Delta_X))$ is a closed subscheme of $G$.  In the proof of \cite[I.1.15(3)]{Jan03}, one considers a map $f$ from ${\rm Sp}_k(R)$ to $\mathfrak{Mor}(Y_i,X)$, where $R$ is an arbitrary $k$-algebra; one obtains a map $f'$ from ${\rm Sp}_k(R)\times Y_i$ to $X$.  We apply this construction when $R= A$ and $f= \delta_i$; then $f'$ is just the map $\beta_i\colon G\times Y_i\to X\times X$.  Let $K_i$ be the ideal of $k[G\times Y_i]= A\otimes_k B_i$ corresponding to $\beta_{Y_i}^{-1}(\Delta_X)$ as before, and let $K_i'$ be the ideal of $A$ corresponding to $\delta_i^{-1}(\mathfrak{Mor}(Y_i,\Delta_X))$.  By the argument of {\em loc.\ cit.}, $K_i'$ is the smallest ideal of $A$ such that $K_i'\otimes B_i$ contains $K_i$.

 Since $Y$ has $(G,\Delta)$-complexity at most $d$ by Lemma~\ref{lem:diagbdd}, $K_i'$ is generated by elements of the form $f^{(i)}_j= \sum_m a^{(i)}_{mj}\otimes b^{(i)}_{mj}$ for some $a^{(i)}_{mj}\in A$ and some $b^{(i)}_{mj}\in B_i$, where each $a^{(i)}_{mj}$ has a representative in $S$ of degree at most $d$.  As in the proof of Theorem~\ref{thm:normbdd}, we can take $K_i'$ to be generated by the elements $a^{(i)}_{mj}$, so $K_i'$ has complexity at most $d$.  Hence $C_G(Y)$ has complexity at most $d$, and we are done.
\end{proof}

\begin{Corollary}
\label{cor:smoothcent}
 Let $d, n\in {\mathbb N}$.  Then there is a prime $p_0= p_0(n,d)$ such that if:
\begin{itemize}
\item $k$ is any field of characteristic $p\geq p_0$;
\item $G$ is any affine algebraic $k$-group of complexity at most $d$; 
\item $X$ is any locally finitely presented, separated $k$-scheme of complexity at most $d$ and $(G, \Delta)$-complexity at most $d$;
\end{itemize} 
then for any closed subscheme $Y$ of $X$, the centraliser $C_{G}(Y)$ is a smooth closed subscheme of $G$.
\end{Corollary}

\begin{proof}
 This follows from Theorem~\ref{gensmooth} and Theorem~\ref{thm:centbdd}.
\end{proof}

\begin{Remark}
      \label{rem:loc_delta_complexity}
       The hypothesis in Corollary~\ref{cor:smoothcent} that $X$ has $(G,\delta)$-complexity at most $d$ does not involve $Y$; this explains why we get better smoothness results for centralisers than for normalisers (cf.\ Example~\ref{rem:normdpndce}).  This hypothesis on $X$ is difficult to check in general, but it clearly holds if $X$ is of finite type, since then we can take $I$ to be finite.  Here is a useful special case.  Let $X$ be affine. Suppose the action $G\times X\to X$ has complexity at most $d$.  Then one can see that $X$ has $(G, \Delta)$-complexity at most $d$.
      \end{Remark}

\begin{Remark}
 Let $G$, $X$ and $Y$ be as above (defined over arbitrary $R$).  Suppose $G$ has complexity at most $d$.  Suppose that for every algebraically closed $R$-field $k$, $Y_k$ is geometrically reduced (which implies the $k$-points of $Y_k$ are dense in $Y_k$) and the singleton  $\{y\}$ has $G_k$-complexity at most $d$ for all $y\in Y(k)$.  Then $C_{G_k}(Y_k)= \bigcap_{y\in Y(k)} N_{G_k}(\{y\})$; in particular, $C_{G_k}(Y_k)$ is a closed subscheme of $G_k$.  Now each $N_{G_k}(\{y\})$ has $G$-complexity at most $d^2$ by Theorem~\ref{thm:normbdd}, so $C_{G_k}(Y_k)$ has complexity at most $d^2$.  Hence by Theorem~\ref{gensmooth} there exists $p_4= p_4(n,d)$ such that if $\Char(k)\geq p_4$ then $C_{G_k}(Y_k)$ is smooth.  Note that we don't need to assume that $X$ is separated in this argument.
\end{Remark}

\subsection{Proof of main theorems and examples}
\label{subsec:proofs}

We can now give a quick proof of our main results.

\begin{proof}[Proof of Theorems~\ref{thm:smoothcent} and \ref{thm:smoothnorm}]
 Let $G$ and $X$ be as in the statement of Theorem~\ref{thm:smoothcent}. Since $G$ is affine and finitely presented, and $X$ is of finite type, there exists $d\in \N$ such that $G$ has complexity at most $d$ and $X$ has $(G, \Delta)$-complexity at most $d$.  Hence $G_k$ has complexity at most $d$ and $X_k$ has $(G_k, \Delta)$-complexity at most $d$ for any algebraically closed $R$-field $k$.  Now Theorem~\ref{thm:smoothcent} follows from Corollary~\ref{cor:smoothcent}.

The proof of Theorem~\ref{thm:smoothnorm} is similar, using Corollary~\ref{cor:smoothnorm} in place of Corollary~\ref{cor:smoothcent}.
\end{proof}

\begin{proof}[Proof of Corollary~\ref{cor:module}]

 The $G$-module $V$ is a finitely presented affine $G$-scheme over $R$.  The proof for centralisers of subspaces follows immediately from Theorem~\ref{thm:smoothcent} applied to $G$ and $V$.
 
 For normalisers we will use Corollary~\ref{cor:smoothnorm}.  Let $\alpha_1,\ldots, \alpha_n$ be generators for $V^*$ as an $R$-module.  The map $\alpha_1\times\cdots \times \alpha_n$ gives an $R$-linear embedding of $V$ as a subspace of $\A^n_R$.  There exists $d\in \N$ such that $G$ has complexity at most $d$ and the map $G\times V\to V$ given by the action has complexity at most $d$.  Now let $k$ be an algebraically closed $R$-field.  Then $G_k$ has complexity at most $d$ and the map $G_k\times V_k\to V_k$ given by the action has complexity at most $d$, by Remark~\ref{rem:base_extn}.  Let $W$ be a subspace of $V_k$.  Then $W$ has complexity at most 1 with respect to the embedding of $V_k$ in $\A^n_k$, so $W$ has $G$-complexity at most $d$ by Remark~\ref{rem:loc_complexity}.  The result now follows from Corollary~\ref{cor:smoothnorm}.
\end{proof}

\begin{Remark}
 We sketch another proof of the normaliser part of Corollary~\ref{cor:module}.  Let $k$ be an algebraically closed $R$-field and let $W$ be a subspace of $V_k$.  Let $W'$ be the annihilator of $W$ in $V_k^*$ and set $r= \dim(W')$.  Then $N_{G_k}(W)= C_{G_k}(x)$, where $x$ is the element of $\P(\Lambda^r(V_k^*))$ corresponding to the line $\Lambda^r(W')$ in the exterior power $\Lambda^r(V_k^*)$, so we can deduce the result by applying Theorem~\ref{thm:smoothcent} to the $G$-module $\Lambda^r(V^*)$.  We leave the details to the reader.
\end{Remark}

\begin{Remark}
\label{rem:normdpndce}
 Any hope to extend Theorem~\ref{thm:smoothnorm} to deal with normalisers of arbitrary closed subschemes of $X$ will fail without first imposing some further hypotheses. For example, \cite[Ex.~11.11]{HS16} gives for each prime $p$ a smooth subgroup $H_p$ of $\GL_3$ over an algebraically closed field of characteristic $p$ such that the normaliser of $H_p$ is non-smooth.  Here we can take $G= X= \GL_3$ over $R= \Z$ with $G$ acting on $X$ by conjugation; Theorem~\ref{thm:smoothnorm} does not apply because our closed subschemes $H_p$ are not of the form $Y_k$ for any fixed closed subscheme $Y$ of $X$.
\end{Remark}

\begin{Remark}\label{nonsmoothex}
 Likewise, Theorem~\ref{thm:smoothcent} fails without some kind of complexity hypothesis.  For example, let $G$ be a split simple and simply-connected group over $\Z$ and $N=V_G(\lambda)$ the Weyl module for $G$ with minuscule highest weight $\lambda$. Then $N_k$ is irreducible for each algebraically closed field $k$. When $\Char k=p>0$, let $M_k=(N_k)^{[1]}$ be the Frobenius twist of $N_k$ through $F\colon G_k\to G_k^{(1)}$; as $G_k\cong F(G_k)=G_k^{(1)}$ we have $M_k$ irreducible too. By irreducibility, $C_{G_k}(m)\subsetneqq G_k$ for any $0\neq m\in M_k$. The $k$-group $G$ being connected and smooth it follows that $\dim_k(C_{G_k}(m))<\dim G_k$, yet $\Lie(G_k)$ is in the kernel of the action on $M_k$. Thus $\dim_k\Lie(C_{G_k}(m))=\dim G_k$; it follows that $C_{G_k}(m)$ is not smooth.  Note that $X$ is of finite type here and $G$ and $X$ are fixed, but the action is not.
\end{Remark}

\begin{Remark} Here is a closely related example of the limits of Theorem~\ref{thm:smoothcent}; this time $X$, $G$ and the $G$-action on $X$ are fixed but $X$ is not of finite type.  Let $R=Z$, $G=\SL_2$ and $X$ be the $G$-scheme not of finite type which is the disjoint union of the $G$-modules $H^0(p)=\Ind_B^G(p)$ for every prime $p$---here $B$ is a Borel subgroup of $G$ and the integer $p$ is a free $Z$-module of rank $1$ on which $B$ acts with weight $p$ through the quotient map to a maximal torus. When $\Char k=p$, the simple socle of the $G_{\bar k}$-module $H^0(p)_{\bar k}$ is isomorphic to a Frobenius twist $L(1)_{\bar k}^{[1]}$ of the natural $2$-dimensional $G_{\bar k}$-module $L(1)_{\bar k}$. If $v$ is a point in the socle of $H^0(p)_{\bar k}$ then its centraliser is a proper subgroup of $G_{\bar k}$, but $v$ is centralised by the whole Lie algebra. We conclude that the centraliser of $v$ is not smooth. In this instance, the action map $\alpha\colon G\times X\to X$ is not $d$-bounded for any $d$.\end{Remark}

\subsection{Modules for reductive groups}
Let $k$ be a field of characteristic $p>0$ and let $G$ be a split reductive $k$-group (which by convention means that it is connected). We recall some of the basic representation theory of $G$ as found in \cite[II.1, II.2]{Jan03}. Let $B$ be a Borel subgroup of $G$, containing a split maximal torus $T$ of $G$. 
This choice defines a set of simple roots of the (reduced) root lattice $\Phi$ of $G$, and a subset of dominant weights $X(T)^+$ of the character lattice $X(T)=\Hom(T,\Gm)$. Moreover, there is a 1--1 correspondence between the dominant weights $\lambda\in X(T)^+$ and the simple $G$-modules $L(\lambda)$. Let $k_\lambda$ denote the  $1$-dimensional $k$-module on which $T$ acts with weight $\lambda$; this is also a $B$-module via the canonical projection $B\to T$. The induced representation $H^0(\lambda):=\Ind_B^G(k_\lambda)$ is finite-dimensional and contains $L(\lambda)$ as its unique simple submodule; i.e., as its socle. 

There is a natural pairing of $X(T)$ with the cocharacter lattice $Y(T)=\Hom(\Gm,T)$ denoted \[\langle \ ,\ \rangle:X(T)\times Y(T)\to \Z; (\lambda,\varphi)\mapsto \lambda\circ\phi.\] 
One can show that any root $\alpha\in \Phi\subseteq X(T)$ gives rise to a coroot $\alpha^\vee\in Y(T)$. Of particular interest is the coroot $\alpha_0^\vee$, where $\alpha_0$ is the highest short root of $\Phi$. For example:
\begin{Lemma}
\label{lem:rep_large_p}
(i) If $\langle\lambda,\alpha_0^\vee\rangle\leq p$ then $H^0(\lambda)=L(\lambda)$;

(ii) If $V$ is a $G$-module such that  $\langle\lambda,\alpha_0^\vee\rangle\leq p$ for all dominant weights $\lambda$ in $V$, then $V$ is semisimple.\end{Lemma}

Say a dominant weight $\lambda$ is \emph{$d$-bounded} if $\langle\lambda,\alpha_0^\vee\rangle\leq d$.

\begin{Proposition}\label{prop:repns} Let $\Phi$ be a (reduced) root system and fix $d\in \N$ an integer. There are primes $p_2$ and $p_3$ with the following properties. If $k$ is a field of  characteristic $p\geq p_2$ (resp.~$p\geq p_3$), if $G$ is any (connected) reductive $k$-group with root system $\Phi$, and if $V$ is a $G$-module whose dominant weights are all $d$-bounded, then the centralisers of all closed subschemes of $V$ in $G$ are smooth and the normalisers of all subschemes of $V$ of complexity at most $d$ are smooth.\end{Proposition}

\begin{proof}We may assume $p_2,p_3\geq d$. Then by Lemma~\ref{lem:rep_large_p} and local finiteness \cite[I.2.13--14]{Jan03}, \[V\cong \bigoplus_{\lambda\in\Lambda} L(\lambda)\cong\bigoplus_{\lambda\in\Lambda} H^0(\lambda),\] for some indexing set $\Lambda$ of dominant weights. Recall that the group $G$ is defined over $\Z$ in the sense that there is an algebraic $\Z$-group $\mathbf G$ such that $\mathbf G_k\cong G$ \cite[II.1.17]{Jan03}; furthermore $\mathbf G$ has a Borel subgroup $\mathbf B$ containing a split maximal torus $\mathbf T$ and such that $\mathbf B_k$ is a Borel subgroup of $G$ and $\mathbf T_k$ is a split maximal torus of $T$. All Borel subgroups of $G$ are $G(k)$-conjugate and all split maximal tori of $B$ are $B(k)$-conjugate, so without loss of generality, $B=\mathbf B_k$ and $T=\mathbf T_k$. Now induction commutes with flat base change \cite[I.3.5(3)]{Jan03}, which implies $H^0(\lambda)=\Ind_{\mathbf B}^{\mathbf G}(k_\lambda)_k$. Thus $V$ is isomorphic to the base change to $k$ of a $\mathbf G$-module $\mathbf V$, which is free as a $\Z$-module, possibly of infinite rank. Moreover as there are only finitely many possible isomorphism classes of the direct summands of $\mathbf V$---there are only finitely many $d$-bounded weights---the complexity of the action map of $\mathbb G$ on $\mathbf V$ is bounded by a function depending just on $\Phi$ and $d$, being that of the maximal summand. Now we are done by an application of Corollaries~\ref{cor:smoothcent} and \ref{cor:smoothnorm}. \end{proof}

\section{\texorpdfstring{Application: the Kostant-Kirillov-Souriau theorem in characteristic $p$}{Application: the Kostant-Kirillov-Souriau theorem in characteristic p}}
\label{S:KKS}
A foundational result in the theory of smooth complex Poisson varieties in Weinstein's symplectic foliation theorem, which states that every such variety decomposes into a disjoint union of its symplectic leaves. In general these are not complex submanifolds and, even when they are, they usually fail to be locally closed for the Zariski topology (see \cite[Remark~3.6(1)]{BG03} for example). Nevertheless there is a large class of Poisson varieties which admit locally closed symplectic leaves. If $G$ is a complex algebraic group with Lie algebra $\g$ then $\g^*$ carries a natural Poisson structure. A fundamental result in symplectic geometry states that the coadjoint orbits coincide with the symplectic leaves of $\g^*$; this is known as the Kostant-Kirillov-Souriau theorem. For semisimple groups this construction is exhaustive in a precise sense: every symplectic homogeneous space is a finite covering of such an orbit (see \cite[\textsection IV.7]{GS77}).

The theory of symplectic varieties in positive characteristic is still in its early stages, although the foundations have been carefully lain in the landmark work of Bezrukavnikov--Kaledin \cite{BK08}, where they classified Frobenius constant quantizations of smooth symplectic varieties. Another notable result is the proof of the formal version of Weinstein's splitting theorem \cite{Tik18}, which decomposes a restricted Poisson variety into a product of a symplectic subvariety and a transverse Poisson slice, in a formal neighbourhood of a point.

Although one cannot define symplectic leaves in positive characteristic, many Poisson varieties decompose into a disjoint union of quasi-affine symplectic subvarieties. This seems to be the most natural replacement for the symplectic foliation in this setting.

When $G$ is an algebraic group over an algebraically closed field of positive characteristic we can ask whether the coadjoint orbits are symplectic subvarieties of $\g^*$. In general the answer is negative, and the failure can be traced back to the fact that the quotient map from a group to an orbit is not always separable. We now demonstrate that the Kostant-Kirillov-Souriau theorem holds for algebraic groups over algebraically closed fields of sufficiently large positive characteristics, using Theorem~\ref{thm:smoothcent}. First, we require a preparatory lemma.  If $G$ is an algebraic $R$-group and $k$ is an algebraically closed $R$-field then we write $\g_k$ for $\Lie(G_k)$.

\begin{Proposition}
      Let $G$ be an algebraic $R$-group.  There exists a prime $p_4\in \N$ with the property that if  $k$ is any algebraically closed field of characteristic $p\geq p_4$, then for all $x\in \g_k$ and all $\chi\in \g_k^*$, the subgroups $C_{G_k}(x)$ and $C_{G_k}(\chi)$ are smooth.
\label{P:toppyplop}
\end{Proposition}

\begin{proof} Since $G$ is finitely presented, we have $A=R[G]=k[x_1,\dots,x_n]/I$ with $I=(f_1,\dots,f_r)$, and suppose $G$ has complexity at most $d$---which implies the $f_i$ have degree at most $d$. We may assume $G$ is non-trivial and so $d\geq 1$. The vanishing ideal $I_1$ at the identity is then finitely generated of complexity at most $d$ and so the Lie algebra $\Lie(G)\cong (I_1/I_1^2)^*$ and the adjoint action $Ad:G\times \Lie(G)\to\Lie(G)$ have bounded complexity as a function just of $d$; see \cite[I.2.4(8)]{Jan03} for a formula. (Certainly $d^4$ will suffice.)Moreover, if $Q$ is any $R$-algebra, $Q[G_Q]\cong R[G]\otimes_R Q$ has complexity at most $d$ and the same formula implies the adjoint action has complexity at most $d^4$ also. A similar argument yields the bounded complexity of the co-adjoint action.   The result now follows from Corollary~\ref{cor:smoothcent}.\end{proof}

Let $G$ be an algebraic $R$-group and let $p_4$ be as Proposition~\ref{P:toppyplop}. Pick an algebraically closed $R$-field $k$ of characteristic $p$. The following is a version of the Kostant-Kirillov-Souriau theorem.

\begin{Theorem}
\label{T:leaves}
If $p\geq p_4$ then the induced Poisson structure on coadjoint orbits in $\g_k^*$ is symplectic. Hence $\g_k^*$ decomposes into a disjoint union of locally closed symplectic subvarieties.
\end{Theorem}

\begin{proof}
Since $p\geq p_3$ it follows from Proposition~\ref{P:toppyplop} that the coadjoint stabilisers in $\g_k$ are all smooth. For the proof we fix $\chi \in \g_k^*$ and write $\O$ for the coadjoint orbit of $\chi$, write $\ad^*$ for the coadjoint representation of $\g_k$ on $\g_k^*$ and write $\g_k^\chi$ for the stabiliser of $\chi$. Thanks to \cite[2.1]{Jan04}, the natural bijective morphism $G_k/C_{G_k}(\chi) \rightarrow \O$ is separable and so we have isomorphisms
\begin{eqnarray}
\label{e:coadjointtangent}
T_\chi \O \overset{\sim}{\longrightarrow} \ad^*(\g_k) \chi \overset{\sim}{\longrightarrow} \g_k / \g_k^\chi.
\end{eqnarray}

Let $I \subseteq k[\g_k]$ be the defining ideal of $\overline{\O}$. Since $\O$ is $G$-stable, $I$ is $G$-stable and hence $\g_k$-stable. Hence $I$ is a Poisson ideal and $k[\overline{\O}]$ inherits a Poisson structure. Let $x_1,...,x_n$ be a basis for $\g_k$. The rank of the Poisson structure at $\chi$ is the rank of the matrix $\pi^\chi$ such that $\pi_{i,j}^\chi = \chi([x_i, x_j])$. However $\pi^\chi$ is nothing other than the matrix of the linear form $\bigwedge^2 \g_k \to k$ given by $(x,y) \mapsto \chi([x,y])$. The radical of this form is $\g_k^\chi$ and so we conclude that the rank of the Poisson structure on $\overline{\O}$ at $\chi$ is $\dim \g_k/ \g_k^\chi$. It follows from \eqref{e:coadjointtangent} that the rank coincides with $\dim \O$. Since $G$ acts by Poisson automorphisms and $\O$ is homogeneous we conclude that the Poisson structure on $\overline\O$ has full rank at every point of $\O$, as desired.
\end{proof}

{\footnotesize
\bibliographystyle{amsalpha}
\bibliography{bib}}

\end{document}